\documentclass[12pt,reqno]{amsart}

\usepackage[utf8]{inputenc}
\usepackage{lmodern}
\usepackage{framed} 
\usepackage{color}
\usepackage{xcolor}
\usepackage{bbm}
\usepackage{mathrsfs}
\usepackage{enumitem}
\usepackage{amsmath}
\usepackage{amssymb}
\usepackage{amsfonts}
\usepackage{latexsym}
\usepackage{amsthm}
\usepackage[english]{babel}
\usepackage[T1]{fontenc}
\usepackage{caption}
\usepackage{pgf,tikz}
\usetikzlibrary{arrows}
\usetikzlibrary{shapes}
\usetikzlibrary{patterns}
\usetikzlibrary{shapes.geometric}
\usepackage{geometry}
\usepackage{graphicx}
\usepackage{fancyhdr}
\usepackage{faktor}
\usepackage{mdframed}
\usepackage{ytableau}
\usepackage{cellspace}
\usepackage{stmaryrd}
\usepackage{yhmath}

\usetikzlibrary{matrix,automata,topaths,decorations.pathreplacing}

\definecolor{Theme}{gray}{0}

\makeatletter

\renewenvironment{leftbar}[1][\hsize]
{%
    \MakeFramed{\hsize#1\advance\hsize-\width\FrameRestore}%
}
{\endMakeFramed}

\renewenvironment{proof}[1][\proofname]{\par
   \pushQED{\begin{center}\textcolor{Theme}{\ensuremath{\blacksquare}}\end{center}}%
  \normalfont \topsep6\p@\@plus6\p@\relax
  \trivlist
  \item[\hskip\labelsep
        \bfseries
    #1\@addpunct{.}]\ignorespaces
}{%
  \popQED\endtrivlist\@endpefalse
}

\newtheoremstyle{Thm}  
  {10pt}   
  {20pt}   
  {\itshape}  
  {0pt}       
  {\bfseries} 
  {}         
  {\newline}  
  {\textcolor{Theme}{\thmname{ #1}\thmnumber{ #2}}\textbf{\thmnote{ (#3)}}} 
  
  \newtheoremstyle{Thmused}  
  {10pt}   
  {20pt}   
  {\itshape}  
  {0pt}       
  {\bfseries} 
  {}         
  {\newline}  
  {\textcolor{Theme}{\thmname{ #1}} \textbf{\thmnote{#3}}} 
  
\newtheoremstyle{Def}  
  {10pt}   
  {20pt}   
  {}  
  {0pt}       
  {\bfseries} 
  {}         
  {\newline}  
  {\textcolor{Theme}{\thmname{ #1} \thmnumber{#2}}\textbf{\thmnote{ (#3)}}} 

  \newtheoremstyle{Rem}  
  {10pt}   
  {20pt}   
  {}  
  {0 pt}       
  {\bfseries} 
  {}         
  {\newline}  
  {\textcolor{Theme}{\thmname{ #1} \thmnumber{#2}}\textbf{\thmnote{ (#3)}}} 

\theoremstyle{Thmused} 
\theoremstyle{Thm} \newtheorem{Thm}{Theorem}
\theoremstyle{Thm} \newtheorem{Lemma}{Lemma}[section]
\theoremstyle{Thm} \newtheorem{Prop}{Proposition}[section]
\theoremstyle{Thm} \newtheorem{Cor}{Corollary}[Prop]
\theoremstyle{Thm} 
\theoremstyle{Thm} 
\theoremstyle{Def} 
\theoremstyle{Def} 
\theoremstyle{Def} 
\theoremstyle{Rem} \newtheorem*{Rem}{Remark}
\theoremstyle{Rem} 

\let\<\langle
\let\>\rangle
\newcommand{\R}{\mathbbm{R}}

\newcommand{\C}{\mathbbm{C}}

\newcommand{\Z}{\mathbbm{Z}}

\newcommand{\1}{\mathbbm{1}}

\newcommand{\wt}{\widetilde}

\newcommand{\SL}{\mathop{\rm{SL}}}
\renewcommand{\O}{\mathrm{O}}
\newcommand{\SO}{\mathop{\rm{SO}}}
\newcommand{\U}{\mathrm{U}}

\newcommand{\SU}{\mathop{\rm{SU}}}

\newcommand{\SOe}{\mathop{\rm{SO_e}}}

\newcommand{\Mat}{\mathop{\rm{Mat}}}
\newcommand{\Res}{\mathop{\rm{Res}}}

\newcommand{\Ima}{\mathop{\rm{Im}}}
\newcommand{\Ker}{\mathop{\rm{Ker}}}

\newcommand{\Tr}{\mathop{\rm{Tr}}}
\newcommand{\End}{\mathop{\rm{End}}}
\newcommand{\Id}{\mathop{\rm{Id}}}
\newcommand{\Ind}{\mathop{\rm{Ind}}}

\newcommand{\triv}{\mathrm{triv}}

\newcommand{\g}{\mathfrak{g}}
\newcommand{\h}{\mathfrak{h}}

\newcommand{\n}{\mathfrak{n}}
\newcommand{\Ak}{\mathfrak{k}}
\renewcommand{\a}{\mathfrak{a}}

\newcommand{\sL}{\mathfrak{sl}}

\newcommand{\E}{\mathscr{E}}
\newcommand{\Resi}{\mathcal{R}}
\newcommand{\Reso}{\mathrm{R}}

\newcommand{\Hi}{\mathscr{H}}

\newcommand{\RA}{\mathbf{R}}
\newcommand{\LA}{\mathbf{L}}
\newcommand{\IWk}{{\textrm{\bf k}}}
\newcommand{\IWH}{{\textrm{\bf H}}}
\newcommand{\IWn}{{\textrm{\bf n}}}

\newcommand{\fonction}[5]{\begin{array}{lrcl}
#1: & #2 & \longrightarrow & #3 \\
    & #4 & \longmapsto & #5 \end{array}}
    
\newcommand{\fracobl}[2]{\raisebox{0.8ex}{$#1$} \Big/ \raisebox{-0.7ex}{$#2$}}

\renewcommand{\textbf}[1]{\begingroup\bfseries\mathversion{bold}#1\endgroup}

\makeatother

\definecolor{Rouge}{RGB}{200,0,0}
\definecolor{Vert}{RGB}{0,147,83}
\definecolor{Bleu}{RGB}{0,89,157}

\title[Resonances of $\square$ on $\SOe(2,2)/\SOe(2,1)$]{Resonances of the d'Alembertian on the Anti-de Sitter space $\SOe(2,2)/\SOe(2,1)$}
\author[Resonances of $\square$ on $\SOe(2,2)/\SOe(2,1)$]{Simon ROBY}
\address{Yau Mathematical Sciences Center, Tsinghua University, Beijing, 100084, China}
\email{roby@mail.tsinghua.edu.cn}
\date{}

\begin{document}

\maketitle

\tableofcontents

\begin{abstract}
We consider the action of the d'Alembertian on the functions on the pseudo-Riemannian 3-dimensional Anti-de Sitter space. We determine the resonances of this operator. With each resonance one can associate a \emph{residue representation}. We give an explicit description of these representations via Langlands parameters. 
\end{abstract}

\section{Introduction}

The resonances has been introduced first in Quantum Mechanics. We can fix the beginning of this notion around 1926, when Schrödinger studied the Stark effect. Roughly speaking they generalize, on non-compact spaces, the eigenvalues of differential operators on compact spaces. For more information about this, one can read for instance \cite{Har07} or \cite{DyaZwo19}.

Let $\SOe(p,q)$ be the connected component of the identity in $\SO(p,q)$, where $p$ and $q$ are positive integers. The purpose of this paper is to calculate the resonances of the d'Alembertian operator on the Anti-de Sitter space $\mathbb{X}:=\fracobl{\SOe(2,2)}{\SOe(2,1)}$. The d'Alembertian operator $\square_\mathbb{X}$ acts on the set $C_c^\infty(\mathbb{X})$ of compactly supported smooth functions on $\mathbb{X}$. The resolvent $\Reso$ of $\square_\mathbb{X}$ is a holomorphic function on the complex complement of the spectrum of $\square_\mathbb{X}$ in the complex plane. If one can find a meromorphic continuation of this function across the spectrum, then the poles of this meromorphic continuation are called the \emph{(scattering) resonances}. 
The study of the resonances of the Laplace operator began on Riemannian symmetric spaces of non-compact type and has been carried out by several authors. The symmetric space has maximal flat subspaces, all of the same dimension, called the rank of the symmetric space. When the rank is one, Guillop\'{e} and Zworski \cite{GuillopeZworski95}, Miatello and Will \cite{MiaWill}, and Hilgert and Pasquale \cite{HilgPasq} computed the resonances with different methods. The two last articles also found the \emph{residue representations} arising from the resonances. For higher rank symmetric spaces, the problem is still open. Partial results were obtained by Mazzeo and Vasy \cite{MaVa05} and Strohmaier \cite{Stro05}. Complete results for most of the rank 2 cases were proved in a series of papers by Hilgert, Pasquale and Przebinda \cite{HiPaPz3,HiPaPz2, HiPaPz}. 

In this article, we determine the resonances and the residue representations of $\square_\mathbb{X}$ acting on $C_c^\infty(\mathbb{X})$. The point is that space $\mathbb{X}$ is not Riemannian any more, but is a \emph{pseudo-Riemannian} symmetric space. This article gives an example of computation of resonances in this context. The general pseudo-Riemannian case (or even the Anti-De Sitter case) is beyong reach at the moment. However, Frahm and Polyxeni obtained first important results in \cite{FraPol23}, where they considered the d'Alembertian operator acting on the pseudo-Riemannian spaces $\fracobl{\U(p,q;\mathbb F)}{\U(1;\mathbb F)\times \U(p-1,q;\mathbb F)}$ (where $\mathbb F$ is either the field of real, complex, quaternionic or octonionic numbers).

In \cite{Andersen01}, a Paley-Wiener type theorem is proved for the Anti-De Sitter symmetric space $\mathbb X$ through the isomorphism of this space with $\SL(2,\R)$. The action of $\SL(2,\R)\times\SL(2,\R)$ on the right and on the left of $\SL(2,\R)$ is equivalent to the action of $\SOe(2,2)$ on $\mathbb X$. This isomorphism makes the action of $\square_{\mathbb X}$ on $C_c^\infty(\mathbb{X})$ correspond to the Casimir of $\SL(2,\R)$ acting on compactly supported functions on $\SL(2,\R)$. This is explained explicitely in Section \ref{section Not and Backg}. In Section \ref{section Resonances}, we determine the resonances. By exploiting the analysis on $\SL(2,\R)$ and in particular the Harish-Chandra-Plancherel Theorem, we can use similar method as the one used in \cite{HilgPasq} to determine them. This analysis can be found in \cite{RobThesis21}. 
In Section \ref{section Residue representations}, we study the residue representations arising from the resonances and we describe them explicitly: Theorem \ref{Thm residue representations} gives the Langlands parameters of the irreducible parts of these representations. The study of these representations cannot be assimilate to what has been done before, because the group acting is not $\SL(2,\R)$, but $\SL(2,\R)\times \SL(2,\R)/\pm \Id\simeq \SOe(2,2)$.

\section{Notations and background}
\label{section Not and Backg}

We shall use the standard notations $\Z_+,~ \Z,~\R, ~\C$ for the nonnegative integers, the integers, the real numbers, the complex numbers. For a complex number $z \in \C$, we denote by $\Re(z)$ and $\Im(z)$ its real and imaginary parts, respectively. We denote by $\RA$ and $\LA$ the usual right and left representations respectively. 

\subsection{Context}
\label{section context} Let $\langle\cdot,\cdot\rangle$ be the bilinear form on $\R^4$ given by
$$\langle x,y\rangle := x_1y_1+x_2y_2-x_3y_3-x_4y_4$$
where $x,y\in \R^4$. Let $\Mat_4(\R)$ be the space $4\times 4$ matrices with real coefficients. We consider the action of $\Mat_4(\R)$ on $\R^4$ by multiplication
$$x\mapsto g\cdot x~~~,~\text{for }g \in {\Mat}_4(\R) \text{ and } x \in \R^4~.$$ 
Let $\O(2,2)$ be the matrix Lie group preserving $\langle\cdot,\cdot\rangle$. We denote by $\wt G :=\SOe(2,2)$ its connected component of the identity. Let $\wt H \subset \wt G$ denote the isotropy subgroup of $(1,0,0,0)\in \R^4$. The group $\wt H$ is then isomorphic to $\SOe(1,2)$. The homogeneous space $\wt G/\wt H$ is the hyperbolic space
$$\mathbb{X}=\{x\in \R^4~|~ \langle x,x\rangle=1\}~.$$
This is a pseudo-Riemannian symmetric space. 
Let $\wt K\simeq \SO(2)\times \SO(2)$ be the maximal compact subgroup of $\wt G$ consisting of the matrices in $\wt G$ with two blocks of $\SO(2)$ on the diagonal. 
Each matrix in $\SO(2)$ is written as
$$k_\theta:=\begin{pmatrix}\cos(\theta)&\sin(\theta)\\-\sin(\theta)&\cos(\theta)\end{pmatrix}~, ~~~~\theta\in \R~.$$

\subsection{An isomorphism with $\SL(2,\R)$}\label{section Iso SL2R}
Let $G=\SL(2,\R)$. The identification $\wt G = G\times G /\pm \Id$ is defined as follows; See for instance in \cite[section 5]{Andersen01}. Let $\Mat_2(\R)$ the set of $2\times 2$ matrices with real coefficients.  
We consider the vector space isomorphism between $\R^4$ and $\Mat_2(\R)$ mapping the canonical basis of $\R^4$ onto the basis consisting of the vectors
$$\begin{pmatrix}1&0\\0&1\end{pmatrix},~\begin{pmatrix}0&-1\\1&0\end{pmatrix},~\begin{pmatrix}0&1\\1&0\end{pmatrix},~\begin{pmatrix}1&0\\0&-1\end{pmatrix}~,$$ namely $$\R^4\ni x=(x_1,x_2,x_3,x_4)\mapsto g_x=\begin{pmatrix}x_1+x_4&-x_2+x_3\\x_2+x_3&x_1-x_4\end{pmatrix}~.$$
We introduce a bilinear form $\langle g_x,g_y\rangle:=\langle x,y\rangle$ on $\Mat_2(\R)$. 
We consider the left and right action of $G\times G$ on $\Mat_2(\R)$:
$$(g_1,g_2)\cdot X := g_1Xg_2^{-1}~,$$
where $(g_1,g_2)\in G\times G \text{ and } X\in {\Mat}_2(\R)$. 
A direct computation shows that 
$$\langle g_x,g_y\rangle:=\langle x,y\rangle=\frac12 \det(g_y)\Tr\big(g_xg_y^{-1}\big)=\frac12 \det(g_x)\Tr\big(g_yg_x^{-1}\big)~.$$
Hence, the form $\langle \cdot,\cdot\rangle$ on $\Mat_2(\R)$ is $G\times G$ invariant. 
The elements $(g_1,g_2)$ and $(-g_1,-g_2)$ act in the same way, so $G\times G/\pm\Id$ can be embedded in the isometry group preserving this bilinear form of signature (2,2), which is $\O(2,2)$. By uniqueness of the component of the identity, we have the Lie group isomorphism $$G\times G/\pm\Id\simeq \wt G~.$$

Moreover $x\in \mathbb{X}$ is equivalent to  $\det(g_x)=1$. Thus the isomorphism $x\mapsto g_x$ restricts to a diffeomorphism $\mathbb{X}$ and $G$.
Consequently we have $C^\infty(G)\simeq C^\infty(\mathbb{X})$. We denote this isomorphism by \begin{equation}\label{eq iso C(G) to C(X)}\begin{array}{rcl}
C^\infty(G) & \longrightarrow & C^\infty(\mathbb{X}) \\
f & \longmapsto & \wideparen{f}~. \end{array}\end{equation}
Via this isomorphism, the two sided action of $G\times G$ on $C^\infty({G})$ \begin{equation}\label{eq action of GxG on function}
    (g_1,g_2)\cdot f(g):=\LA (g_1)\cdot \RA (g_2)\cdot f(g) = f(g_1^{-1}gg_2)~~, \text{ for all }g,g_1,g_2 \in {G}~,
\end{equation}
corresponds to the action on the left of $\wt G$ on $C^\infty(\mathbb{X})$.

In spherical coordinates, each element $x\in \mathbb{X}$ is given by
$$x=x(t,\psi_1,\psi_2) = (\cosh t \cos\psi_1,\cosh t \sin\psi_1,\sinh t \sin\psi_2,\sinh t \cos\psi_2)\in \mathbb{X}~,$$ 
where $t,\psi_1,\psi_2\in \R$.
Let $\theta= \frac{\psi_1-\psi_2}2$ and $\phi= \frac{\psi_1+\psi_2}2$. 
Through the map $x\mapsto g_x$, the corresponding element in $G$ is  $$g(t,\theta,\phi):= g_{x(t,\phi+\theta,\phi-\theta)} = k_\theta a_t k_\phi = \begin{pmatrix}\cos(\theta)&\sin(\theta)\\-\sin(\theta)&\cos(\theta)\end{pmatrix}\begin{pmatrix}e^t&0\\0&e^{-t}\end{pmatrix}\begin{pmatrix}\cos(\phi)&\sin(\phi)\\-\sin(\phi)&\cos(\phi)\end{pmatrix}~.$$

We denote by $\Omega$ the Casimir operator of $\g:=\sL(2,\R)$, the Lie algebra of $G$. Consider the  ‘$KAK$’ decomposition above for $g(t,\theta,\phi)\in G$. Then for all $f\in C_c^\infty(G)$ (see \cite[Chapter 6, Lemma 26]{Varadarajan89})
\begin{multline}
    \RA(\Omega)f = \Bigg[\frac{\partial^2}{\partial t^2}+ 2\frac{\cosh(2t)}{\sinh(2t)}\frac{\partial}{\partial t}\\
    -\frac1{4\cosh^2(t)}\left(\frac{\partial^2}{\partial \theta^2}+\frac{\partial^2}{\partial \phi^2}+2\frac{\partial^2}{\partial \theta\partial\phi}\right)
    +\frac1{\sinh^2(t)}\left(\frac{\partial^2}{\partial \theta^2}+\frac{\partial^2}{\partial \phi^2}-2\frac{\partial^2}{\partial \theta\partial\phi}\right) \Bigg]f~.
\end{multline}
which is the formula for the d'Alembertian $\square_\mathbb{X}=\partial_{x_1}^2+\partial_{x_2}^2-\partial_{x_3}^2-\partial_{x_4}^2$ on $\mathbb{X}$ in spherical coordinates (see \cite[page 70]{Andersen01}).

\section{Resonances of the Laplace operator on $\SL(2,\R)$}\label{section Resonances}

We recall that $G=\SL(2,\R)$ is isomorphic to $\SU(1,1)$ and locally isomorphic to $\wt H$. The resonances of the Laplace operator for the group $G$ are known. One can find them for the line bundles over the real hyperbolic plane, with the residue representations in \cite{Will}. The different method we will give here can be found for $G$ in the unpublished parts II.1.2 and II.2.2 in \cite{RobThesis21}. We use a direct method based on the original Harish-Chandra~-~Plancherel formula for $G$. The ideas are based on \cite{HilgPasq}.

Through the isomorphism given in section \ref{section Iso SL2R}, the resonances of the Laplacian of $G$ gives directly the resonances of the d'Alembertian $\square_\mathbb{X}$. The residue representations will be different than those of $G$, because the action is not the one of $G$ but that of $G\times G$, since we consider the action of $\wt G$ on $C_c^\infty(\mathbb{X})$. They will be given in section \ref{section Residue representations}.

Recall that irreducible unitary representations of $K:=\SO(2)$ are the characters $\chi_n :K \mapsto \C$ defined by
\begin{equation*}
    \chi_n(k_\theta)=e^{in\theta}  
\end{equation*}
with $\theta\in\R$ and $n\in\Z$.
Let $A$ and $N$ be defined as follows
\begin{equation*}
   A=\left\{a_t:=\left(
\begin{array}{ccc}
e^t & 0\\
0 & e^{-t}
\end{array}\right)~\Big|~ t\in\R \right\}\,,\ \  
N=\left\{\left(
\begin{array}{ccc}
1 & y\\
0 & 1
\end{array}\right)~\Big|~ y\in\R\right\}\,.
\end{equation*}
with Lie algebras
\begin{equation*}
   \a=\left\{\left(
\begin{array}{ccc}
t & 0\\
0 & -t
\end{array}\right)~\Big|~ t\in \R\right\}\,,\ \  
\n=\left\{\left(
\begin{array}{ccc}
0 & y\\
0 & 0
\end{array}\right)~\Big|~ y\in\R\right\}\,.
\end{equation*}
The system of (restricted) roots of the pair $(\g,\a)$ is $\{\pm\alpha\}$ where $\alpha\begin{pmatrix}
t & 0\\
0 & -t
\end{pmatrix} =2t$. We choose $\{\alpha\}$ as positive system such that the half sum of positive roots is $\rho := \frac{1}{2}\alpha$. We define $e^{\lambda\alpha}(a_t) := e^{2\lambda t}$, for all $\lambda \in \C$. 
The centralizer of $A$ in $K$ is $M = \{\pm I_2\}$. We denote by $0$ is the trivial representation and 
$1$ is the non-trivial representation of $M$. By $(\pi^\delta_{i\lambda},\Hi^\delta_{i\lambda})$ we mean the principal series representation \mbox{$\Ind_{MAN}^{G}(\delta\otimes e^{i\lambda\rho} \otimes \triv_N)$}, where $\lambda \in \C$ and $\delta\in\hat M=\{0,1\}$. The space $\Hi_{i\lambda}^{\delta}$ is the Hilbert space completion of the space of smooth functions $f: G\rightarrow \C$ such that
\begin{equation*}
    f(manx)=a^{(1+i\lambda)\rho}\delta(m) f(x) = t^{1+i\lambda}\delta(m) f(x) \,,
\end{equation*}
where $m\in M$, $a = \begin{pmatrix}
t & 0\\
0 & t^{-1}
\end{pmatrix}\in A$, $n\in N$, $x\in {G}$. The representation $\pi^\delta_{i\lambda}$ is defined, for all $g\in G$ and $f\in \Hi^\delta$, by 
$\pi^\delta_{i\lambda}(g)f = \RA(g)f$. 

Let \begin{equation}\label{eq delta(l)}
    \delta(l)=\left\{\begin{array}{cc}
    0 & \text{ if } l \text{ is odd} \\
    1 & \text{ if } l \text{ is even}
\end{array}\right..
\end{equation}

\begin{Prop}[\cite{LangSL2R}, VI. \S 6, Theorem 8, pages 119-124]\label{Prop representations of SL2}
If $i\lambda$ is not an integer, then
\begin{equation*}
    {\Hi^0_{i\lambda}}|_K=\bigoplus_{n\in 2\Z}{\Ind}_K^{G}(\chi_n)~~~~~~~~ \text{and}~~~~~~~~ \Hi^{1}_{i\lambda}|_K=\bigoplus_{n\in 2\Z+1}{\Ind}_K^{G}(\chi_n)~.
\end{equation*}

If $\lambda=0$, then
\begin{equation*}
    \Hi^0_{0}=\bigoplus_{n\in 2\Z}{\Ind}_K^{G}(\chi_n) ~~~~~~~~ \text{and}~~~~~~~~ \Hi^1_{0}=\bigoplus_{\substack{n\in 2\Z+1\\n\geq1}}{\Ind}_K^{G}(\chi_n) ~~\oplus~~ \bigoplus_{\substack{n\in 2\Z+1\\n\leq-1}}{\Ind}_K^{G}(\chi_n)\,.
\end{equation*}

If $m\geq 2$ is an integer and $i\lambda=m-1$, then $\Hi^{\delta(m-1)}_{m-1}$ contains two irreducible submodules,
\begin{equation*}
    D^m=\bigoplus_{\substack{m\leq n\\ n-m\in 2\Z}}{\Ind}_K^{G}(\chi_n)\,,~~~~~~
    D^{-m}=\bigoplus_{\substack{-m\geq n\\ n-m\in 2\Z}}{\Ind}_K^{G}(\chi_n)\,.
\end{equation*}
The quotient module ($\Hi^{\delta(m-1)}_{m-1}$ divided by the two submodules) is irreducible, finite dimensional of dimension $m-1$.

The module $\Hi^{\delta(m)}_{m-1}$ is irreducible and corresponds to
$$\bigoplus_{ n-m\in 2\Z+1}{\Ind}_K^{G}(\chi_n)\,$$

If $m\geq 2$ is an integer and $i\lambda=-m+1$, then $\Hi^{\delta(m-1)}_{m-1}$ contains the finite dimensional submodule
\[
\C e_{-m+2}\oplus \C e_{-m+4}\oplus ... \oplus \C e_{m-2}\,. 
\]
where $e_n$ is the function on ${G}$ defined by 
\begin{equation}\label{e_n}
    e_n\left(\left(
\begin{array}{ccc}
u & 0\\
0 & u
\end{array}
\right)
\left(
\begin{array}{ccc}
y & x\\
0 & 1
\end{array}
\right)
\left(
\begin{array}{ccc}
\cos \theta & \sin\theta\\
-\sin\theta & \cos \theta
\end{array}
\right)\right)
=y^{1+i\lambda}e^{in\theta}\,.
\end{equation}
Here $u, y$ and $x$ are real such that the indicated product is in ${G}$. 
Thus we have the highest weight modules, lowest weight modules, finite dimensional modules and modules with unbounded $K$-types on both sides.

The $(\g,K)$-modules $\Hi^\delta_{m-1}$ and $\Hi^\delta_{-m+1}$ are dual to each other.

Here is the complete list of the irreducible unitarizable $(\g,K)$-modules.
\begin{enumerate}
\item Lowest weight module $D^m$ with lowest weight $m\geq 1$ and the highest weight module $D^m$ with highest weight $m\leq -1$\;
\item Principal series $\Hi^0_{i\lambda}$ and $\Hi^{1}_{i\lambda}$, $\lambda\in\R\setminus \{0\}$;
\item Principal series $\Hi^0_{0}$;
\item Complementary series $\Hi^0_{s}$, $0<s<1$;
\item Trivial representation.
\end{enumerate}
\end{Prop}

We denote by $\Theta^\delta_{i\lambda}$ the character of the representation of $\pi^\delta_{i\lambda}$.
We also define the convolution product \begin{equation*}
    \left(\Theta^\delta_{i\lambda}\ast f \right)(g) = \int_{G}\Theta^\delta_{i\lambda}(x)f(gx) ~dx~.
    \end{equation*}
The Harish-Chandra's Plancherel theorem for $G$ may be stated as follows:
\begin{Thm}[Inversion formula for $G$]\label{Plancherel formula for SL2}
For any $f\in C_c^\infty({G})$,
\begin{multline}\label{eq Plancherel formula for SL2}
2\pi f(g)=\sum_{n=1}^\infty n\Big(\Tr\big(D^{(n+1)}+D^{(-n-1)}\big)\ast f\Big)(g)\\
+\frac{1}{2}\int_0^\infty ~\left(\Theta^0_{i\lambda}\ast f\right)(g)~~\lambda \tanh\left(\frac{\pi\lambda}{2}\right)\,d\lambda\\
+ \frac{1}{2}\int_0^\infty \left(\Theta^1_{i\lambda}\ast f\right)(g)~~\lambda \coth\left(\frac{\pi\lambda}{2}\right)\,d\lambda~.
\end{multline}
\end{Thm}

\begin{Rem}
\begin{leftbar}
As an operator acting on $L^2$, the Laplacian of $\SL(2,\R)$ (or equivalently, the d’Alembertian) has a discrete and a continuous spectrum. The discrete part consists of eigenvalues, which are poles of the resolvent of the Laplacian as operator on $L^2$. On the other hand, the resonances originate from the continuous part of the spectrum and this is the reason why we shall determine them by first projecting the operator onto its continuous spectrum. This seems to be consistent with the scattering interpretation of resonances. 
In the physics literature, the physical states of a system show up as bound states or as resonance states. The bound states correspond to the eigenvalues of the Hamiltonian of the system as operator acting on $L^2$. Suppose, as in our case, that the Hamiltonian admits a meromorphic extension when it is restricted to smooth compactly supported functions. The extension is usually defined on a Riemann surface above $\C$ (or a portion of $\C$). The original complex plane containing the $L^2$ spectrum— and hence the eigenvalues— is called the physical sheet. The resonances are the poles of the meromorphic extension that are located outside of the physical sheet. In many cases (and in our case too), the Riemann surface is a square root function with two sheets. The second one, called the unphysical sheet, is hence where the resonances can be found. A resonances might mirror an eigenvalue (in the sense that they are corresponding points in the two sheets). However as we shall see Theorem 2, none of the residue representations we determine is in $L^2$, as a resonant state should be. 
For further information on the physical interpretation of resonances, we refer to \cite{DyaZwo19}, \cite{His12}. See also \cite[Section 3]{Stro05} for the Riemannian symmetric space situation. 

Thus, from now on, we consider the projection of the resolvent on the continuous part of \eqref{eq Plancherel formula for SL2}. 

\end{leftbar}
\end{Rem}

A direct computation (see \cite[page 195]{LangSL2R}) shows that in the universal enveloping algebra $U(\g_c)$
\begin{equation*}
    \Omega = \begin{pmatrix}1&0\\0&1\end{pmatrix}^2 + \begin{pmatrix}0&1\\1&0\end{pmatrix}^2 - \begin{pmatrix}0&1\\-1&0\end{pmatrix}^2
\end{equation*} 
and 
\begin{equation*}
    \pi_{i\lambda}(\Omega) = -(\lambda^2 +1) \Id~.
\end{equation*}
We consider $\Omega$ acting by the left (or the right) on $C_c^\infty({G})$.
We have, for all $x\in \mathbb{X}$
\begin{equation}\label{eq iso Dal Casimir}
    \square_\mathbb{X} \left(\wideparen{\Theta^\delta_{i\lambda}\ast f}\right)(x) = \LA(\Omega)\left(\Theta^\delta_{i\lambda}\ast f\right)(g_x) = -(\lambda^2 +1) \left(\Theta^\delta_{i\lambda}\ast f\right)(g_x)~,
\end{equation}
where the isomorphism $f\mapsto \wideparen f$ has been defined in \eqref{eq iso C(G) to C(X)}.

Set \begin{equation}
    \Reso(z) := (-\LA(\Omega) - z)^{-1}~.
\end{equation}
Then, up to a constant, for $f\in C_c^\infty({G})$ and for all $z\in \C\setminus [1,+\infty[$ and $g \in {G}$
\begin{align*}
    \left(\Reso(z)f\right)(g) &= \int_0^\infty (\lambda^2+1-z)^{-1} ~\left(\Theta^0_{i\lambda}\ast f\right)(g)~~\lambda \tanh\left(\frac{\pi\lambda}{2}\right)\,d\lambda\\
&~~~+ \int_0^\infty (\lambda^2+1-z)^{-1}\left(\Theta^1_{i\lambda}\ast f\right)(g)~~\lambda \coth\left(\frac{\pi\lambda}{2}\right)\,d\lambda~.\\
\end{align*}
Choose the holomorphic branch of the square root function is such that $\sqrt{-1} = i$. Changing the variable $\zeta = \sqrt{z - 1}$, we get for $$\zeta\in\C^+:= \left\{z\in\C~|~\Im(z) >0\right\}~,$$ up to a constant,
\begin{align}\label{eq Resolvent for SL2 tanh}
    \left(\Reso(\zeta^2+1)f\right)(g) &=\int_0^\infty (\lambda^2-\zeta^2)^{-1} ~\left(\Theta^0_{i\lambda}\ast f\right)(g)~~ \lambda\tanh\left(\frac{\pi\lambda}{2}\right)\,d\lambda\\\label{eq Resolvent for SL2 coth}
&~~~+ \int_0^\infty (\lambda^2-\zeta^2)^{-1} \left(\Theta^1_{i\lambda}\ast f\right)(g)~~ \lambda\coth\left(\frac{\pi\lambda}{2}\right)\,d\lambda~,
\end{align}
We will deal with each integral \eqref{eq Resolvent for SL2 tanh} and \eqref{eq Resolvent for SL2 coth} separatly. For the first integral \eqref{eq Resolvent for SL2 tanh} we split it as follows:
\begin{align*}
    \int_0^\infty (\lambda^2-\zeta^2)^{-1} ~\left(\Theta^0_{i\lambda}\ast f\right)(g)~~ \lambda\tanh\left(\frac{\pi\lambda}{2}\right)\,d\lambda&=\frac12\int_0^\infty \left(\frac{1}{\lambda -\zeta} +\frac{1}{\lambda +\zeta}\right) ~\left(\Theta^0_{i\lambda}\ast f\right)(g)~~ \tanh\left(\frac{\pi\lambda}{2}\right)\,d\lambda\\ &= \frac12\int_\R \frac{1}{\lambda -\zeta} ~\left(\Theta^0_{i\lambda}\ast f\right)(g)~~ \tanh\left(\frac{\pi\lambda}{2}\right)\,d\lambda
\end{align*}
By Morera's theorem, the resulting function is holomorphic in $\zeta \in \C\setminus \R$. We want to extend it meromorphically beyond the upper half-plane. 
For this, we shift the contour of integration in the direction of the negative imaginary axis and apply the residue theorem. This gives, for any $y\in \R\setminus\Z_+$
\begin{multline}
        \int_0^\infty (\lambda^2-\zeta^2)^{-1} ~\left(\Theta^0_{i\lambda}\ast f\right)(g)~~ \lambda\tanh\left(\frac{\pi\lambda}{2}\right)\,d\lambda = \\\frac12\int_{\R-iy} \frac{1}{\lambda -\zeta} ~\left(\Theta^0_{i\lambda}\ast f\right)(g)~~ \tanh\left(\frac{\pi\lambda}{2}\right)\,d\lambda\\
        - \pi i\sum_{\substack{l\in 2\Z_++1\\ l< y}}\frac{1}{-il -\zeta} ~\left(\Theta^{0}_{l}\ast f\right)(g)~~ \Res_{\lambda=-il}\tanh\left(\frac{\pi\lambda}{2}\right)~.
\end{multline}
The second term \eqref{eq Resolvent for SL2 coth} does not work the same way, because $\lambda\mapsto\coth\left(\frac{\pi\lambda}{2}\right)$ has a pole at $\lambda=0$. Thus we shift the contour first. For a fixed $\zeta\in \C^+$ et $y\in \R\setminus2\Z$, $y> \Im(\zeta)$, the residue theorem gives
\begin{multline*}
    \frac12\int_\R (\lambda^2-\zeta^2)^{-1} \left(\Theta^1_{i\lambda}\ast f\right)(g)~~ \lambda\coth\left(\frac{\pi\lambda}{2}\right)\,d\lambda=\frac12
    \int_{\R-yi} (\lambda^2-\zeta^2)^{-1} \left(\Theta^1_{i\lambda}\ast f\right)(g)~~ \lambda\coth\left(\frac{\pi\lambda}{2}\right)\,d\lambda\\
    -\pi i\sum_{\substack{l\in 2\Z^*_+\\ l< y}}\frac{-il}{-l^2 -\zeta^2} ~\left(\Theta^{1}_{l}\ast f\right)(g)~~ \Res_{\lambda=-il}\coth\left(\frac{\pi\lambda}{2}\right)\\
    -\frac{\pi i}2 \left(\Theta^1_{i\zeta}\ast f\right)(g)~~ \Res_{\lambda=-\zeta}\left[\frac1{\lambda+\zeta}\coth\left(\frac{\pi\lambda}{2}\right)\right]~.
\end{multline*}
Now we can split the first term as before
\begin{align*}
    \int_{\R-yi} (\lambda^2-\zeta^2)^{-1} \left(\Theta^1_{i\lambda}\ast f\right)(g) &\lambda\coth\left(\frac{\pi\lambda}{2}\right)\,d\lambda\\&=\frac12\int_{\R-yi} \left(\frac1{\lambda-\zeta}+\frac1{\lambda+\zeta}\right) \left(\Theta^1_{i\lambda}\ast f\right)(g) \coth\left(\frac{\pi\lambda}{2}\right)\,d\lambda
    \\&=\frac12\int_{\R-yi} \frac1{\lambda-\zeta} \left(\Theta^1_{i\lambda}\ast f\right)(g) \coth\left(\frac{\pi\lambda}{2}\right)\,d\lambda\\ & + \frac12 \int_{\R-yi} \frac1{\lambda+\zeta} \left(\Theta^1_{i\lambda}\ast f\right)(g) \coth\left(\frac{\pi\lambda}{2}\right)\,d\lambda
\end{align*}
In the second integral, we change the variable to $-\lambda$
\begin{equation*}
\int_{\R-yi} \frac1{\lambda+\zeta} \left(\Theta^1_{i\lambda}\ast f\right)(g) \coth\left(\frac{\pi\lambda}{2}\right)\,d\lambda=\int_{\R+yi} \frac1{\lambda-\zeta} \left(\Theta^1_{i\lambda}\ast f\right)(g) \coth\left(\frac{\pi\lambda}{2}\right)\,d\lambda~,
\end{equation*}
and we shift again to $\R-yi$ passing through the singularities $\lambda=\zeta$ and the ones of $\lambda\mapsto \coth\left(\frac{\pi\lambda}{2}\right)$
\begin{multline*}
    \frac12\int_{\R+yi} \frac1{\lambda-\zeta} \left(\Theta^1_{i\lambda}\ast f\right)(g) \coth\left(\frac{\pi\lambda}{2}\right)\,d\lambda=\frac12\int_{\R-yi} \frac1{\lambda-\zeta} \left(\Theta^1_{i\lambda}\ast f\right)(g) \coth\left(\frac{\pi\lambda}{2}\right)\,d\lambda\\
    -\pi i\sum_{\substack{l\in 2\Z\\ |l|< y}}\frac{1}{il -\zeta} ~\left(\Theta^{1}_{-l}\ast f\right)(g)~~ \Res_{\lambda=il}\coth\left(\frac{\pi\lambda}{2}\right)\\
    -\pi i \left(\Theta^1_{i\zeta}\ast f\right)(g)~~ \Res_{\lambda=\zeta}\left[\frac1{\lambda-\zeta}\coth\left(\frac{\pi\lambda}{2}\right)\right]
\end{multline*}
Summing everything in the original formula of \eqref{eq Resolvent for SL2 coth}, several residues are canceled and we get
\begin{multline*}
    \int_\R (\lambda^2-\zeta^2)^{-1} \left(\Theta^1_{i\lambda}\ast f\right)(g)~~ \lambda\coth\left(\frac{\pi\lambda}{2}\right)\,d\lambda=\int_{\R-yi} \frac1{\lambda-\zeta} \left(\Theta^1_{i\lambda}\ast f\right)(g)~~ \coth\left(\frac{\pi\lambda}{2}\right)\,d\lambda\\
    +2i\sum_{\substack{l\in 2\Z_+^*\\ l< y}}\frac{1}{il +\zeta} ~\left(\Theta^{1}_{l}\ast f\right)(g)+\frac{i}{\zeta}\left(\Theta^{1}_{0}\ast f\right)(g)
\end{multline*}

The following proposition is obtained putting \eqref{eq Resolvent for SL2 tanh} and \eqref{eq Resolvent for SL2 coth} together. 

\begin{Prop}\label{Prop Resonances SL2R}
\begin{multline}\label{meromorphiccontinuation}
    \left(\Reso(\zeta^2+1)f\right)(g) = \frac12\int_{\R-iy} \frac{1}{\lambda -\zeta} ~\left(\Theta^0_{i\lambda}\ast f\right)(g)~~ \tanh\left(\frac{\pi\lambda}{2}\right)\,d\lambda\\
 +\frac12\int_{\R-iy} \frac{1}{\lambda -\zeta} ~\left(\Theta^1_{i\lambda}\ast f\right)(g)~~ \coth\left(\frac{\pi\lambda}{2}\right)\,d\lambda\\
    +\frac{i}{\zeta}\left(\Theta^{1}_{0}\ast f\right)(g)+ 2i\sum_{\substack{l\in \Z^*_+\\ l< y}}\frac{1}{il +\zeta} ~\left(\Theta^{\delta(l)}_{l}\ast f\right)(g) ~,
\end{multline}
where $\delta(\cdot)$ is defined in \eqref{eq delta(l)}.
The above formula give a meromorphic continuation of the resolvent of the positive Laplace operator on the half-plane above $\R-iy$ with simple poles at $\zeta = -il$ where $l\in \Z_+,~ l< y$. 
We make $y\rightarrow +\infty$. Let $S=\{(z,\zeta) \in \C^2~|~z = \zeta^2+1\}$. Then the resolvent extends meromorphically from $S^+ = \{(z,\zeta) \in M~|~\Im(\zeta)>0\}$ to $S$ and with simple poles at $(-l+1,-il)$, $l\in \Z_+$. These poles are the resonances of the Laplace operator.
\end{Prop}

By the isomorphism in \eqref{eq iso Dal Casimir}, we have

\begin{Cor}\label{Prop Resonances Dal}
The resonances of $\square_\mathbb{X}$ on $C_c^\infty(\mathbb{X})$ are given by Proposition \ref{Prop Resonances SL2R}. 
\end{Cor}


\section{Residue representations}\label{section Residue representations}

Recall the two sided action of $G\times G$ on $C^\infty({G})$ given by \eqref{eq action of GxG on function}. 
The residues at $\lambda_l:=-il$, $l\in \Z_+$, in the meromorphic continuation \eqref{meromorphiccontinuation} span a $G\times G$-invariant subspace of $C^\infty(G)$. By the previous isomorphism, this is exactly the image of the ${G}\times {G}$-intertwining map:
\begin{equation}\label{residueoperatordefinition}
    \fonction{\Resi_l}{C_c^\infty(G)}{C^\infty(G)}{f}{\Theta^{\delta(l)}_{l}\ast f~,}
\end{equation}
where $\delta(\cdot)$ is defined in \eqref{eq delta(l)}. We denote this space by
\begin{equation}\label{residuerepresentationdefinition}
    \E_l := \{\Theta^{\delta(l)}_{l}\ast f ~|~ f \in  C_c^\infty(G)\}~.
\end{equation}

\subsection{$K\times K$-isotypic components}

By definition, the $K\times K$-isotypic components of $\E_l$ are the spaces of the functions $u\in C^\infty(G)$ such that for all $(k_1,k_2)\in K\times K$ and $g\in G$, 
\begin{equation}\label{eq K types}
    (k_1,k_2)\cdot u(g) = u(k_1^{-1}gk_2) =  \chi_n(k_1)\chi_m(k_2^{-1}) u(g)
\end{equation}
for some $n,m\in \Z$. They are called \emph{$(\chi_n,\chi_m)$-spherical functions} in the literature. We denote the set of such functions by $\mathcal{S}_{n,m}$. 

From now on, we use the compact picture of the principal series representations. It is obtained by restriction of the elements of $\Hi^\delta_{i\lambda}$ to $K$. Its representation space, which we denote by $\Hi^\delta$, is the Hilbert completion of: 
\begin{equation*}
    \{f: K \rightarrow \Hi_\delta ~|~ f(mk) = \delta(m)f(k) ~\text{ for all }k \in K,~ m\in M \}
\end{equation*}
with respect to $L^2$ inner product. It is independent of $\lambda\in \C$. Define 
the Iwasawa decomposition on $g\in G$ decomposition by:
\begin{equation}
    g=\IWn(g)e^{\IWH(g)}\IWk(g)~.
\end{equation}
Then the action of $\pi^\delta_{i\lambda}$ is given by:
\begin{equation*}
    \pi^\delta_{i\lambda}(g)f(k) := e^{(i\lambda+1)\rho(\IWH(kg))} f\big(\IWk\big(kg)\big)
\end{equation*} 
for all $g \in G$, $k \in K$ and $f\in \Hi^\delta$.
The representation $\pi^\delta_{i\lambda}$ is unitary for $\lambda \in \R$. We denote by $\langle\cdot,\cdot \rangle$ the inner product on this space.

\begin{Lemma}\label{Lemma KxK types}
    Let $f\in C^\infty_c(G)\cap \mathcal{S}_{n,m}$.
Then $\Theta^{\delta(l)}_{l}\ast f\equiv0$ if $n,m\not \in 2\Z+\delta(l)$. Moreover if $n,m\in 2\Z+\delta(l)$, then $\Theta^{\delta(l)}_{l}\ast f\in \mathcal{S}_{n,m}$.
\end{Lemma}

\begin{proof}
Recall that the adjoint of the principal series on $\Hi^\delta$ is given by\begin{equation*}
    \langle\pi^{\delta}_{{\lambda}}(x)\chi_i, \chi_j\rangle = \langle \chi_i, \pi^{\delta}_{-\overline{\lambda}}(x^{-1})\chi_j\rangle~,
\end{equation*}
for all $\delta\in \hat{M}$, $\lambda\in \C$, $x\in G$ and $i,j\in \Z$. 
    The lemma is shown by the following direct computations:
\begin{align*}
    \Theta^{\delta(l)}_{l}\ast f(g)
    =& \int_{K}\int_{G}\Theta^{\delta(l)}_{l}(x)f(kk^{-1}gx)dxdk\\
    =& \int_{K}\int_{G}\Theta^{\delta(l)}_{l}(g^{-1}kx)f(kx)dxdk\\
    =& \int_{K}\int_{G}\Theta^{\delta(l)}_{l}(kxg^{-1})f(kx)dxdk\\
    =& \int_{K}\int_{G}\sum_{j\in2\Z+\delta(l)}\langle\pi^{\delta(l)}_{l}(kxg^{-1})\chi_j, \chi_j\rangle f(kx)dxdk\\
    =& \int_{K}\int_{G}\sum_{j\in2\Z+\delta(l)}\langle\pi^{\delta(l)}_{l}(xg^{-1})\chi_j, \pi^{\delta(l)}_{-\overline{l}}(k^{-1})\chi_j\rangle f(kx)dxdk\\
    =& \int_{G}\sum_{j\in2\Z+\delta(l)}\langle\pi^{\delta(l)}_{l}(xg^{-1})\chi_j, \chi_j\rangle\left[\int_{K}\chi_n(k^{-1})\overline{\chi_j(k^{-1})}dk\right] f(x)dx\\
\end{align*}
and 
\begin{align*}
    \Theta^{\delta(l)}_{l}\ast f(g)=& \int_{G}\Theta^{\delta(l)}_{l}(x)\int_{K}f(gxk^{-1}k)dkdx\\
    =& \int_{G}\int_{K}\sum_{j\in2\Z+\delta(l)}\langle\pi^{\delta(l)}_{l}(xk)\chi_j, \chi_j\rangle f(gxk)dkdx\\
    =& \int_{G}\sum_{j\in2\Z+\delta(l)}\langle\pi^{\delta(l)}_{l}(x)\chi_j, \chi_j\rangle\left[\int_{K}\chi_m(k^{-1})\chi_j(k)dk\right] f(gx)dx~.\\
\end{align*}
Thus if $n,m\not\in 2\Z+\delta(l)$, the convolution is zero.
Since $$\Theta^{\delta(l)}_{l}\ast f(g) = \int_{G}\Theta^{\delta(l)}_{l}(x)f(gx)dx = \int_{G}\Theta^{\delta(l)}_{l}(x)f(xg)dx~, $$
we see that $\Theta^{\delta(l)}_{l}\ast f\in \mathcal{S}_{n,m}$
\end{proof}

\subsection{$G\times G$ action}

Now we have to know how $G\times G$ acts on $(\chi_n,\chi_m)$-spherical functions.

\begin{Rem}[Method by direct computations]
\begin{leftbar}
Let $\g_\C$ be the complexification of $\g$ and denote by $\mathcal{U}(\g_\C)$ its universal enveloping algebra.
We recall the action of $\g$ by the principal series representations on the functions $\chi_n\in \Hi^{\delta}$. We choose the following basis of $\g_\C$. 
Let $$E=\begin{pmatrix}1&i\\i&-1\end{pmatrix}~~,~~
               F=\begin{pmatrix}-1&i\\i&1\end{pmatrix}~~,~~
               D=\begin{pmatrix}0&1\\1&0\end{pmatrix}\in \Ak~.$$
One can show that the infinitesimal action of $\pi_l^{\delta(l)}$ is given by 
\begin{equation}
    \begin{array}{cl}\displaystyle
      \pi_l^\delta(E)\chi_n &= (n+l+1)\chi_{n+2}~,\\
    \pi_l^\delta(F)\chi_n&= (n-l-1)\chi_{n-2}~,\\
    \pi_l^\delta(D)\chi_n &= in\chi_n~.
    \end{array}
\end{equation}

Let $X\in \g_\C$. Suppose $f$ is a $(\chi_n,\chi_m)$-spherical function. 
Similar computations as those for Lemma \ref{Lemma KxK types} show that 
\begin{equation}
    \RA(X)\cdot \left(\Theta^{\delta(l)}_{l}\ast f\right)\in \left\{\begin{array}{cl}
    \mathcal{S}_{n,m+2} & \text{ if } X=E \text{ and } l-m-1\ne0~,
    \\
    \mathcal{S}_{n,m-2} & \text{ if } X=F\text{ and } l+m-1\ne0~,
\end{array}\right.
\end{equation}
and 
\begin{equation}
    \LA(X)\cdot \left(\Theta^{\delta(l)}_{l}\ast f\right)\in \left\{\begin{array}{cl}
    \mathcal{S}_{n-2,m} & \text{ if } X=E \text{ and } l- n+1\ne0~,
    \\
    \mathcal{S}_{n+2,m} & \text{ if } X=F \text{ and } l +n+1\ne0~.
\end{array}\right.
\end{equation}
It seems not easy here to prove that $\mathcal{S}_{n,m}\cap\E_l$ is nonzero. This is why we use an other method. 
\end{leftbar}
\end{Rem}

We decompose the map $\Resi_l$ which sends $f\in C_c^\infty(G)$ onto $\E_l$ by $f\mapsto \Theta^{\delta(l)}_{l}\ast f$ as follows.  The actions of $G$ on the right and on the left commute, so the invariant subspaces of the left action are stable by the right action and vice-versa. We will do this decomposition for each of these actions.

Define for every representation $(\pi,V_\pi)$, for every $f\in C^\infty_c(G)$, 
\begin{equation}
    \pi(f):= \int_G \pi(x)f(x) dx ~~\in \End(V_\pi).
\end{equation}

For each $n\in 2\Z+\delta(l)$, we define one truncating map $\mathcal{T}_{l,n}$, defined by \begin{equation}
        \fonction{\mathcal{T}_{l,n}}{C^\infty_c(G)}{\Hi^{\delta(l)}}{f}{\pi_l^{\delta(l)}(f)\chi_n~.},
    \end{equation} We denote also by $\mathcal{P}_{l,n}$ the maps 
\begin{equation}
    \fonction{\mathcal{P}_{l,n}}{\Hi^{\delta(l)}}{C^\infty(G)}{\phi}{\left(g\mapsto\left\langle\pi_l^{\delta(l)}(g^{-1}) \phi,\chi_n\right\rangle\right)~.}
\end{equation}
These maps are also called \emph{Poisson transforms} in the literature. We denote, for all $n\in 2\Z+\delta(l)$, by $\E_{l,n}$ the image of the composition of these two maps $\mathcal{P}_{l,n}\circ \mathcal{T}_{l,n}~.$ So \begin{equation}\label{eq Eln}
    \E_{l,n}\simeq\fracobl{\Ima\left(\mathcal{T}_{l,n}\right)}{\Ker\left(\mathcal{P}_{l,n}\right)\cup\Ima\left(\mathcal{T}_{l,n}\right)}
\end{equation}
Then we sum, for all $n\in 2\Z+\delta(l)$, the functions $\mathcal{P}_{l,n}\circ\mathcal{T}_{l,n}(f)$ to get 
    \begin{equation*}\label{eq residue map as the sum of Tl,n}
        \Theta^{\delta(l)}_{l}\ast f(g) = \sum_{n\in 2\Z+\delta(l)}\mathcal{P}_{l,n}\circ\mathcal{T}_{l,n}(f)(g)~.
    \end{equation*}
    In other words, we decompose $\Resi_l$ in the two maps
    \begin{equation}
        \begin{array}{ccc}
         C_c^\infty(G) & \longrightarrow &\bigoplus\limits_{n\in 2\Z+\delta(l)}\Hi^{\delta(l)} \\
         f                  & \longmapsto& (\ldots,\mathcal{T}_{l,\delta(l)-2}(f),\mathcal{T}_{l,\delta(l)}(f),\mathcal{T}_{l,\delta(l)+2}(f),\ldots)~,
    \end{array}
    \end{equation}
    and 
    \begin{equation}
        \begin{array}{ccc}
          \bigoplus\limits_{n\in 2\Z+\delta(l)}\Hi^{\delta(l)} &\longrightarrow & \E_l\\
          (\ldots,\phi_{\delta(l)-2},\phi_{\delta(l)},\phi_{\delta(l)+2},\ldots)                    &\longmapsto &  \sum\limits_{n\in 2\Z+\delta(l)}\mathcal{P}_{l,n}\phi_n  ~.
    \end{array}
    \end{equation}

\begin{Lemma}\label{Lemma image T}
For every fixed $n\in 2\Z+\delta(l)$, the map $\mathcal{T}_{l,n}$ intertwines $(\LA,C_c^{\infty}(G))$ and $(\pi^{\delta(l)}_{l},\Hi^{\delta(l)})$. 
The range of $\mathcal{T}_{l,n}$ is the closure of the space spanned by $\{\pi_l^{\delta(l)}(g)\chi_n~|~g\in G\}$ in $(\pi_l^{\delta(l)},\Hi^{\delta(l)})$. 
\end{Lemma}

\begin{proof}
     Fix $n\in 2\Z+\delta(l)$ and $f\in C_c^{\infty}(G)$. We have, for all $g\in G$
    \begin{align*}
        \mathcal{T}_{l,n}\left(\LA(g)f\right)& = \int_{G}\big(\LA(g)f\big)(x)\pi_l^{\delta(l)}(x)\chi_n dx\\
                                   & = \int_{G}f(x) \pi_l^{\delta(l)}(gx)\chi_n dx\\
                                   & = \pi_l^{\delta(l)}(g)\mathcal{T}_{l,n}(f)~.
    \end{align*}
    So $\mathcal{T}_{l,n}$ is an intertwining map and for all $k \in K$ and $f\in\mathcal{S}_{m_1,m_2}$:
    \begin{equation*}
        \pi_l^{\delta(l)}(k)\mathcal{T}_{l,n}(f)=\mathcal{T}_{l,n}\left(\LA(k)f\right) = \chi_{m_1}(k)\mathcal{T}_{l,n}(f)~.
    \end{equation*}
    But 
    \begin{align*}
        \mathcal{T}_{l,n}(f)&=\int_{G}f(x)\pi_l^{\delta(l)}(x)\chi_n dx\\
                            &=\int_{G}\int_{K\times K}f(kxh)\pi_l^{\delta(l)}(kxh)\chi_n dkdhdx\\
                            &=\int_{G}\left(\int_{K}\chi_{m_2}(h^{-1})\chi_n(h)dh\right)~\left(\int_{K}\chi_{m_1}(k)\pi_l^{\delta(l)}(kx)\chi_n dk\right)~f(x)dx\\
                            &=\delta_{m_2,n} \int_{G}~\left(\int_{K}\chi_{m_1}(k)\pi_l^{\delta(l)}(kx)\chi_n dk\right)~f(x)dx~,
    \end{align*}
    where $\delta$ is the Krönecker symbol. So $\mathcal{T}_{l,n}\left(\mathcal{S}_{m_1,m_2}\right) = \{0\}$ if $m_2 \ne n$. 
    Using the same proof as \cite[Lemma 4.1]{ROBY2021} we have also that $\mathcal{T}_{l,n}(C_c^{\infty}(G))$ is the closed subspace of $\Hi^{\delta(l)}$ spanned by the left translates by $G$ of $\chi_n$. This gives the the space spanned by $\{\pi_l^{\delta(l)}(g)\chi_n~|~g\in G\}$ in the principal series $(\pi_l^{\delta(l)},\Hi^{\delta(l)})$ (see Proposition \ref{Prop representations of SL2}). 
\end{proof}

\begin{Lemma}\label{Lemma Kern P}
    For every $n\in 2\Z+\delta(l)$, the map $\mathcal{P}_{l,n}$ intertwines $(\pi^{\delta(l)}_{l},\Hi^{\delta(l)})$ and $(\LA,C_c^{\infty}(G))$. The Kernel of $\mathcal{P}_{l,n}$ is the smallest quotient of $(\pi^{\delta(l)}_{l},\Hi^{\delta(l)})$ containing $\chi_n$. 
\end{Lemma}

\begin{proof}
Let $n\in 2\Z+\delta(l)$ and $\phi\in \Hi^{\delta(l)}$. We have, for all $g,x\in G$ 
    \begin{align*}
        \mathcal{P}_{l,n}(\pi_l^{\delta(l)}(g)\phi)(x)&= \left\langle\pi_l^{\delta(l)}(x^{-1}) \pi_l^{\delta(l)}(g)\phi,\chi_n\right\rangle\\
                            &= \left\langle\pi_l^{\delta(l)}\left((g^{-1}x)^{-1}\right) \phi,\chi_n\right\rangle\\
                            &=\mathcal{P}_{l,n}(\phi)(g^{-1}x) = \LA(g)\left(\mathcal{P}_{l,n}(f)\right)(x)
    \end{align*}
So $\mathcal{P}_{l,n}$ is an intertwining map. 
Moreover, for all $m\in 2\Z+\delta(l)$, $k_1,k_2\in K$ and $g\in G$, 
\begin{equation*}
    \mathcal{P}_{l,n}(\chi_m)(k_1gk_2) = \left\langle\pi_l^{\delta(l)}(k_2^{-1}g^{-1}k_1^{-1}) \chi_m,\chi_n\right\rangle = \overline{\chi_m(k_1)}\overline{\chi_n(k_2)}\mathcal{P}_{l,n}(\chi_m)(g)~.
\end{equation*}
So $\mathcal{P}_{l,n}(\chi_m)\in \mathcal{S}_{m,n}$, but it could be zero. Recall the composition series in Proposition \ref{Prop representations of SL2} of $(\pi_l^{\delta(l)},\Hi_l^{\delta(l)})$:
\begin{center}
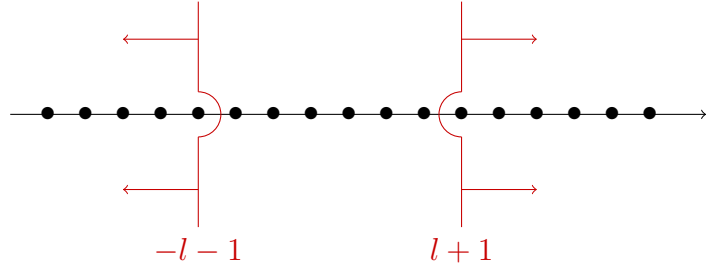

\begin{tikzpicture}[scale=1]
\draw [->](-5,0) to (4.25,0);
\draw [Rouge,-](-2.5,-1.5) to (-2.5,-0.3);
\draw [Rouge,-](-2.5,0.3) to (-2.5,1.5);
\draw [Rouge](-2.5,-0.3) arc (-90:90:0.3);
\draw [Rouge,->](-2.5,1) to (-3.5,1);
\draw [Rouge,->](-2.5,-1) to (-3.5,-1);
\draw [Rouge,-](1,-1.5) to (1,-0.3);
\draw [Rouge,-](1,0.3) to (1,1.5);
\draw [Rouge](1,-0.3) arc (270:90:0.3);
\draw [Rouge,->](1,1) to (2,1);
\draw [Rouge,->](1,-1) to (2,-1);
\draw (1,-1.5) node [Rouge, below]{$l+1$};
\draw (-2.5,-1.5) node [Rouge, below]{$-l-1$};
\foreach \i in {-9,...,7}{\draw (\i/2,0) node{$\bullet$};}
\end{tikzpicture}
\captionof{figure}{Composition series of $\pi_l^{\delta(l)}$}\label{Fig Composition series of pi l}
\end{center}
In Figure \ref{Fig Composition series of pi l}, the bullets represent the $K$-types $\Ind_K^G\chi_m$, for $m\in 2\Z+\delta(l)$ and the action $\pi_l^{\delta(l)}(G)$ can just pass the red walls in the sense of the arrows. 
Thus, $\mathcal{P}_{l,n}(\chi_m)= \left\langle\pi_l^{\delta(l)}(g^{-1}) \chi_m,\chi_n\right\rangle = 0$ if and only if \begin{equation}
    \begin{array}{rl}
        n \in [l+1,+\infty[ & \text{ and } m\in ]-\infty,-l-1],  \\
         \text{or } n \in [-l+1,l-1] & \text{ and } m\in ]-\infty,-l-1]\cup [l+1,+\infty[, \\
         \text{or } n \in ]-\infty,-l-1] & \text{ and } m\in [l+1,+\infty[~.\\
    \end{array}
\end{equation}
Thus the kernel of $\mathcal{P}_{l,n}$ is the smallest quotient of $(\pi_l^{\delta(l)},\Hi_l^{\delta(l)})$ containing $\chi_n$. 
\end{proof}

Denote for all $m\in 2\Z+\delta(l)$, \begin{equation}
    [m]=\left\{j\in2\Z+\delta(l)~|~ \chi_j  \text{ is in the same component as } \chi_m \text{ in }  (\pi_l^{\delta(l)},\Hi_l^{\delta(l)})\right\}
\end{equation}
We mean by \emph{components} the spaces delimited by the barriers without regarding the arrows, in other words, the irreducible subquotients. 
From \eqref{eq Eln}, it follows that:

\begin{Prop}\label{Prop Eln repr}
The representation $\left(\LA,\E_{l,n}\right)$ is equivalent to the smallest subquotient of the principal series $(\pi^{\delta(l)}_l,\Hi^{\delta(l)}_l)$ containing $\chi_n$. It is included in the closure of the space spanned by 
\begin{equation}
    \bigoplus_{j\in [n]} \mathcal{S}_{j,n}~.
\end{equation}
\end{Prop}

One can prove similar results to Lemmas \ref{Lemma image T}, \ref{Lemma Kern P} and \ref{Prop Eln repr} for the right action of $G$ on $\E_l$ and the principal series $(\pi^{\delta(l)}_{-l},\Hi^{\delta(l)})$ replacing $\mathcal{T}_{l,n}$ and $\mathcal{P}_{l,n}$ respectively by 
\begin{equation}
    \fonction{\mathcal{T}^\RA_{l,n}}{C_c^\infty(G)}{\Hi^{\delta(l)}}{f}{\displaystyle\int_{G}\pi_{-l}^{\delta(l)}(x^{-1})\chi_n \overline{f(x)}dx~,}
\end{equation}
and 
\begin{equation}
    \fonction{\mathcal{P}^\RA_{l,n}}{\Hi^{\delta(l)}}{C_c^\infty(G)}{\phi}{\left(g\mapsto\left\langle\pi_l^{\delta(l)}(g^{-1})\chi_n,\phi\right\rangle\right)~.}
\end{equation}
That is:

\begin{Lemma}\label{Lemma image T R}
For all $n\in 2\Z+\delta(l)$, the map $\mathcal{T}^\RA_{l,n}$ is intertwining between $(\RA,C_c^{\infty}(G))$ and $(\pi^{\delta(l)}_{-l},\Hi^{\delta(l)})$. 
The range of $\mathcal{T}^\RA_{l,n}$ is the closure of the space spanned by $\{\pi_{-l}^{\delta(l)}(g)\chi_n~|~g\in G\}$ in $(\pi_{-l}^{\delta(l)},\Hi^{\delta(l)})$
\end{Lemma}

\begin{Lemma}\label{Lemma Kern P R}
    For all $n\in 2\Z+\delta(l)$, the map $\mathcal{P}^\RA_{l,n}$ is intertwining between $(\pi^{\delta(l)}_{-l},\Hi^{\delta(l)})$ and $(\RA,C_c^{\infty}(G))$. The Kernel of $\mathcal{P}^\RA_{l,n}$ is the smallest quotient of $(\pi^{\delta(l)}_{-l},\Hi^{\delta(l)})$ containing $\chi_n$. 
\end{Lemma}

\begin{Prop}\label{Prop Eln repr R}
The representation $\left(\RA,\E_{l,n}\right)$ is equivalent to the smallest subquotient of the principal series $(\pi^{\delta(l)}_{-l},\Hi^{\delta(l)})$ containing $\chi_n$. It is included in the closure of the space spanned by 
\begin{equation}
    \bigoplus_{j\in [n]} \mathcal{S}_{n,j}~.
\end{equation}
\end{Prop}

The principal series representations of $G\times G$ consist of tensor products of principal series representations of $G$ (see also \cite[II. \S 5]{Przebindathesis} for the irreducible unitary representations). 
Propositions \ref{Prop Eln repr} and \ref{Prop Eln repr R} give the image of $\Resi_l$ as the sum of three components where the $\mathcal{S}_{n,n}$ are, for all $n\in 2\Z+\delta(l)$.
One can remark that it's the smallest components containing the ‘‘diagonal’’ of $\mathcal{S}_{n,n}$, with $n\in 2\Z+\delta(l)$. Recall that we studied functions on the quotient $G\times G/\mathrm{diag}(G)\simeq G$ (see \cite[page 94]{Andersen01}). 

Let us draw a picture for a better understanding (see Figure \ref{Fig Composition series}). One bullet with abscissa $n$ and ordinate $m$ represents the $K\times K$-type $\chi_n\otimes \chi_m$ of $\E_l$ in $\mathcal{S}_{n,m}$. We drew this line in green. 
The gray lines with arrows represent the barriers of the composition series of the representation: the action  $(\LA\otimes\RA, G\times G)$ can only send an element through this line in the sense of the arrows attached to this line. The representation $\E_l$ is direct the sum of three irreducible subquotients. In the picture, they consists of the $K\times K$-types which are not slashed in red. 

\definecolor{light-gray}{gray}{0.6}
\begin{center}
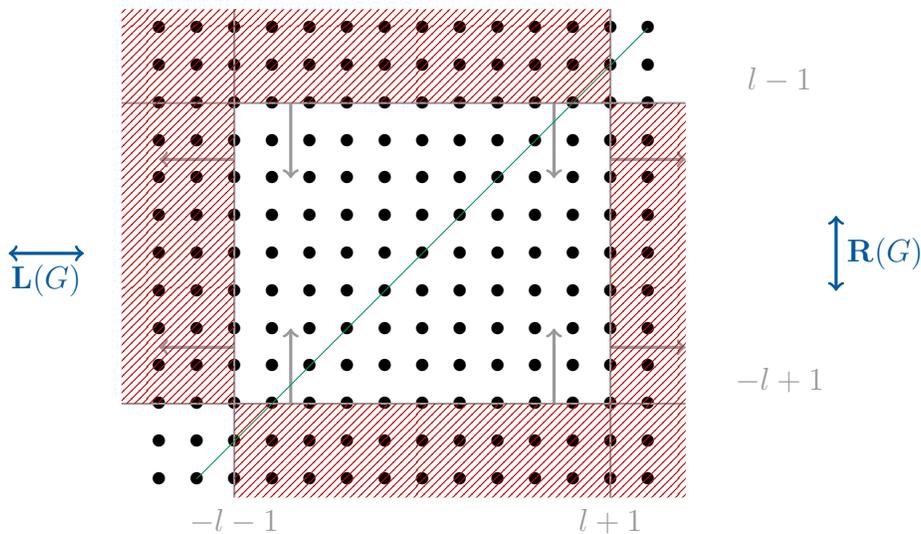

\begin{tikzpicture}[scale=1]
\foreach \j in {-6,...,6}{\foreach \i in {0,...,13}{\draw (-4+\i/2,\j/2) node{$\bullet$};}}
\draw [-,thick,light-gray](-4.5,2) to (3,2);
\draw [->,very thick,light-gray](-2.25,2) to (-2.25,1);
\draw [->,very thick,light-gray](1.25,2) to (1.25,1);
\draw (4.25,2) node[above] {\color{light-gray}$l-1$};
\draw [-,thick,light-gray](-4.5,-2) to (3,-2);
\draw [->,very thick,light-gray](-2.25,-2) to (-2.25,-1);
\draw [->,very thick,light-gray](1.25,-2) to (1.25,-1);
\draw (4.25,-2) node[above] {\color{light-gray}$-l+1$};
\draw [-,thick,light-gray](-3,-3.25) to (-3,3.25);
\draw [->,very thick,light-gray](-3,1.25) to (-4,1.25);
\draw [->,very thick,light-gray](-3,-1.25) to (-4,-1.25);
\draw (-3,-3.25) node[below] {\color{light-gray}$-l-1$};
\draw [-,thick,light-gray](2,-3.25) to (2,3.25);
\draw [->,very thick,light-gray](2,1.25) to (3,1.25);
\draw [->,very thick,light-gray](2,-1.25) to (3,-1.25);
\draw (2,-3.25) node[below] {\color{light-gray}$l+1$};
\draw [<->,very thick,Bleu](-6,0) to (-5,0);
\draw (-5.5,0) node[below] {\color{Bleu}$\LA(G)$};
\draw [<->,very thick,Bleu](5,-0.5) to (5,0.5);
\draw (5,0) node[right] {\color{Bleu}$\RA(G)$};
\draw [-,Vert] (-3.5,-3) to (2.5,3);
\fill [pattern=north east lines,pattern color=Rouge] (-4.5,-2) to (-3,-2) to (-3,2) to (2,2) to (2,3.25) to (-4.5,3.25);
\fill [pattern=north east lines,pattern color=Rouge] (3,2) to (2,2) to (2,-2) to (-3,-2) to (-3,-3.25) to (3,-3.25);
\end{tikzpicture}
\captionof{figure}{Composition series of the residue representation}\label{Fig Composition series}
\end{center}

In Figure \ref{Fig Composition series} we chose $l=9$. For $l=0$, the finite dimensional representation in the middle disappears and one has just two infinite dimensional components in $\E_l$. 

We introduce some notations to give the result in terms of representations of $\wt G$. 
Let $\wt P := \wt M\wt A\wt N$ be the minimal parabolic subgroup of $\wt G$ given by 
\begin{equation}
    \begin{array}{c}
      \wt A:= \{\mathrm{diag}(a_1,a_2,a_1^{-1},a_2^{-1})~|~ a_1,a_2\in \R_+\}~;
      \wt M:= \{\mathrm{diag}(\varepsilon_1,\varepsilon_2,\varepsilon_1,\varepsilon_2)~|~\varepsilon_1,\varepsilon_2=\pm 1\}~;\\
      \mathrm{Lie}(\wt N) = \wt\n:= \left\{{\tiny\begin{pmatrix}
    0&x&0&y\\0&0&-y&0\\0&0&0&0\\0&0&-x&0
\end{pmatrix}}~|~x,y\in \R\right\}~.
    \end{array}
\end{equation}
The representations of $\wt M$ are in $\hat{M}\otimes\hat M$. 

Let $\wt P_1 := \wt M_1\wt A_1\wt N_1$ be the other parabolic subgroup of $\wt G$ given by
\begin{equation}
    \begin{array}{c}
      \wt A_1:= \{\mathrm{diag}(a,a,a^{-1},a^{-1})~|~ a\in \R_+\}~;
      \wt M_1:= \{\mathrm{diag}(g,({}^tg)^{-1})~|~g\in G\}~;\\
      \mathrm{Lie}(\wt N_1) = \wt\n_1:= \left\{{\tiny\begin{pmatrix}
0&y\\0&0
\end{pmatrix}}~|~y\in K\right\}~.
    \end{array}
\end{equation}

There is no other parabolic subgroup. See \cite[2.2.2]{Przebindathesis}. The attentive reader has observed that these parabolic subgroups are not in $\wt G$ for our definition in Section \ref{section context}, but they are the parabolic subgroups of the group isomorphic to $\wt G$ via the conjugation by 
$$\frac{\sqrt{2}}2\begin{pmatrix}
    I_2&I_2\\I_2&-I_2
\end{pmatrix}.$$
See \cite[1.1.9]{Przebindathesis}. 
One Cartan subgroup in $\wt P$ is $\wt H:= \wt M\wt A$ and one in $\wt P_1$ is $\wt H_1:= \wt T_1\wt A_1$ where $\wt T_1$ is the Cartan subgroup of $\wt M_1$ defined by 
\begin{equation}
    \wt T_1 : = \{\mathrm{diag}(k,k)~|~k\in K\}~.
\end{equation}

\begin{Thm}\label{Thm residue representations}
    The residue representations $\E_l$ of the group $\wt G$ arising from the resonances of the d'Alembertian, acting on $C_c^\infty(\mathbb{X})$ is the sum of three irreducible unitarisable representations if $l\ne0$. One is finite dimensional with Langlands parameters $(\wt P,e^{l\wt\rho},\delta(l)\otimes\delta(l))$. The other two are infinite dimensional and mutually contragredient. 
    The Langlands parameters of one of them are $(\wt P_1,\triv_{A_1},\chi_{2l+2})$. 
    The representation $\E_0$ is the sum of the two infinite dimensional representations given for $l=0$. 
\end{Thm}

\begin{proof}
    Suppose first that $l\ne 0$. $\E_l$ is the sum of three irreducible representations.
One can check that the finite dimensional representation is exactly the unique irreducible quotient of the principal series representation $(\pi^{\delta(l)\otimes\delta(l)}_{l},\Hi^{\delta(l)\otimes\delta(l)})$ induced from $\delta(l)\otimes\delta(l)\in \hat{\wt M}$ and $e^{l\wt\rho}\in \hat{\wt A}$, where $\wt\rho$ is the half sum of the positive roots for the group $\wt G$. The representation $\pi^{\delta(l)\otimes\delta(l)}_{l}$ can be seen as the same as Figure \ref{Fig Composition series} but with the grey horizontal lines shifted respectively from $l-1$ and $-l+1$ to $l+1$ and $-l-1$, and reversing the arrows of these lines. 

The two infinite dimensional representations are mutually contragredient. Let us find the Langlands parameters of the top right corner in Figure \ref{Fig Composition series}. We denote it by $\eta_l$. We consider here also the case ‘$l=0$’. 
If it's not a discrete series representation, it is isomorphic to a Langland quotient (see \cite[Theorem 8.54]{Kna1}). We have to find the parabolic subgroup $P$ and its representation $\sigma$ (also called Langlands parameters) corresponding to this Langlands quotient. The minimal $\wt K$-types of $\eta_l$ and $\Ind_P^{\wt G}\sigma$ have to be the same, in terms of Vogan norm (see \cite[Sections 4.1,4.2,6.5,6.6]{VoganRRRLG}). 
The minimal $\wt K$-type of $\eta_l$ is $\chi_{l+1}\otimes \chi_{l+1}$. In \cite[section 2.5]{Przebindathesis} is listed all the possible Langlands parameters for irreducible admissible representations of $\wt G$. The only with minimal $K$-type $\chi_{l+1}\otimes \chi_{l+1}$ (named $\pi_{l+1,l+1}$ there) is the parameters attached to $\wt P_1$, $\chi_{2l+2}\in \hat{\wt {T_1}}$ and $e^{\nu}\in \hat{\wt {A_1}}$ for $\nu\in \a_\C^*$, with $\Re\nu>0$ or $\Re\nu=0$ and $\Im\nu\geq0$ (see \cite[2.5.24.]{Przebindathesis}).
We denote by $\alpha_c$ and $\alpha_{nc}$ the fundamental weights attached respectively to $\wt T_1$ and $\wt A_1$, and by $\epsilon_1, \epsilon_2$ the ones attached to $\wt H$. 
We know that the infinitesimal character on $\wt \h = \mathrm{Lie}(\wt H)$ is $(l+1)\epsilon_1 +(-l-1) \epsilon_2$ (it's just the derivative of $e^{(l+1)\alpha}\otimes e^{(-l-1)\alpha}$, the two roots $2\epsilon_1, 2\epsilon_2$ being real and restrcting to $\alpha$). We then apply the Cayley transform from $\wt H$ to $\wt H_1$ associated to $\epsilon_1-\epsilon_2$. It sends the root $\epsilon_1-\epsilon_2$ on the compact root $2\alpha_c$. So $(l+1)\epsilon_1 +(-l-1) \epsilon_2$ is sent on $2(l+1)\alpha_c$. This is the infinitesimal character of $\Ind_{\wt M_1\wt A_1 \wt N_1}^{\wt G}(\chi_{2l+2}\otimes e^{\nu}\otimes 1)$ when $\nu=0$. So the real part $\nu$ from $\a^*_\C$ has to be $0$. 
\end{proof}

\bibliographystyle{alpha}
\bibliography{Biblio}
\end{document}